\documentclass[a4paper,12pt,reqno]{amsart}
\usepackage{pxfonts}
\usepackage{color,graphicx}
\usepackage[top=1truein,bottom=1truein,left=1truein,right=1truein]{geometry}

\usepackage[numbers,sort&compress]{natbib}
\usepackage{upgreek}
\usepackage{hyperref}
\newtheorem{theorem}{Theorem}[section]
\newtheorem{proposition}[theorem]{Proposition}

\newtheorem{lemma}[theorem]{Lemma}
\theoremstyle{definition}

\newtheorem{remark}[theorem]{Remark}
\newtheorem{assumption}[theorem]{Assumption}

\newenvironment{nouppercase}{%
  \renewcommand{\uppercasenonmath}[1]{}}{}
\makeatletter
\@addtoreset{equation}{section}

\makeatother

\title[]{\Large
Asymptotically Efficient Estimation of Ergodic Rough Fractional Ornstein-Uhlenbeck Process under Continuous Observations
}

\author[K.~Chiba]{\large Kohei Chiba}
\address[Kohei Chiba]{Graduate School of Engineering Science, Osaka University, 3 Machikaneyama-cho 1-Chome, Toyonaka, Osaka, Japan}
\email{kchiba.es@osaka-u.ac.jp}

\author[T.~Takabatake]{\large Tetsuya Takabatake}
\address[Tetsuya Takabatake]{Department of Economics, Hiroshima University, 2-1 Kagamiyama 1-Chome, Higashi-Hiroshima, Hiroshima, Japan}
\email{tkbtk@hiroshima-u.ac.jp}

\date{\today}
\begin{document}

\begin{nouppercase}
	\maketitle
\end{nouppercase}

\begin{abstract}
	We consider the problem of asymptotically efficient estimation of drift parameters of the ergodic fractional Ornstein-Uhlenbeck process under continuous observations when the Hurst parameter $H<1/2$ and the mean of its stationary distribution is not equal to zero.
	In this paper, we derive asymptotically efficient rates and variances of estimators of drift parameters and prove an asymptotic efficiency of a maximum likelihood estimator of drift parameters.
	\end{abstract}
\vspace{0.2cm}

\section{Introduction}
Consider a one-dimensional stochastic differential equation~(SDE) of the form
\begin{equation}\label{definition_SDE}
	\mathrm{d}X_{t}^{\theta}=a(X_{t}^{\theta};\theta)\,\mathrm{d}t+\sigma \,\mathrm{d}B_{t}^{H},\ \ X_0^{\theta}=x_{0}\ \ t\in[0,T],
\end{equation}
where $x_{0}\in\mathbb{R}$, $\sigma>0$ and $B^{H}$ is a fractional Brownian motion~(fBm) with Hurst parameter $H\in(0,1)$.
Here $\theta$ is a parameter to be estimated and we aim at estimating $\theta$ based on a sample path of the solution $(X_{t}^{\theta})_{t\in[0,T]}$ of the SDE~(\ref{definition_SDE}), that is, continuous observations, when a length of observation period $T\to\infty$. 
In this situation, we can assume both parameters $H$ and $\sigma$ are known without loss of generality because their parameters are completely determined by a sample path of $(X_{t}^{\theta})_{t\in[0,T]}$ for any fixed $T>0$.

In this paper, we will investigate the problem of asymptotically efficient estimation of drift parameters of the ergodic fractional Ornstein-Uhlenbeck~(fOU) process under the following two parameterizations:
\begin{align}
	&a(x;\widetilde{\theta})=-\alpha x+\gamma,\ \ \widetilde{\theta}=(\alpha,\gamma)\in(0,\infty)\times\mathbb{R},\label{parameterization2}\\
	&a(x;\theta)=-\alpha(x-\mu),\ \ \theta=(\alpha,\mu)\in(0,\infty)\times\mathbb{R}.\label{parameterization1}
\end{align}
It is obvious that the parameters $\gamma$ and $\mu$ are mutually determined by the equality $\alpha\mu=\gamma$. 
In the case that the mean of the stationary distribution of the ergodic fOU process is equal to zero, that is, $\mu=\gamma=0$, it is well-known that for all $H\in(0,1)$, the MLE of the mean-reverting coefficient $\alpha$ is consistent, asymptotically normal with the convergence rate $\sqrt{T}$ and asymptotically efficient in the Fisher sense under continuous observations, see \cite{Kleptsyna-Le_Breton-2002,Kutoyants-2004-book,Brouste-Kleptsyna-2010}.  
In the other case, \cite{Lohvinenko-Ralchenko-2017} proved that the convergence rate of the MLE of $\gamma$ is $T^{1-H}$, which depends on the Hurst parameter and is slower than the convergence rate $\sqrt{T}$ of the MLE of $\alpha$, when $H>1/2$. 
Moreover, \cite{Tanaka-Xiao-Yu-2019,Chiba-2020} recently showed that the convergence rate of the MLE of $\gamma$ is $\sqrt{T}$ as same as the one of the MLE of $\alpha$ when $H<1/2$. 
Therefore, the convergence rate of the MLE of $\gamma$ no longer depends on the Hurst parameter $H$ when $H<1/2$. 

Our 
main contributions in this paper are (1)~to prove an asymptotic normality of the MLE of $\mu$, (2)~to derive asymptotically efficient rates and variances of estimating $\mu$ and $\gamma$ 
and (3)~to prove the asymptotic efficiency of the MLEs of $\mu$ and $\gamma$ when $H<1/2$.
Our findings are that 
asymptotically efficient rates of estimating $\mu$ and $\gamma$ are different 
and the asymptotically efficient rate of estimating $\mu$ is faster than the one of estimating $\gamma$ when $H<1/2$.
Namely, we obtain the following asymptotic minimax lower bounds of estimators of $\mu$ and $\gamma$ respectively:
\begin{align}
	&\varliminf_{T\to\infty}\sup_{|\theta-\theta_{0}|<r}E_{\theta}^{T}[T(\widehat{\gamma}_{T}-\gamma)^2]\geq \frac{2\gamma_{0}^{2}}{\alpha_{0}},\label{lower_bound2}\\
	&\varliminf_{T\to\infty}\sup_{|\theta-\theta_{0}|<r}E_{\theta}^{T}[T^{2(1-H)}(\widehat{\mu}_{T}-\mu)^2]\geq\frac{\sigma^{2}\lambda_{H}}{\alpha_{0}^{2}}\label{lower_bound1}
\end{align}
for any $r>0$ and any sequences of estimators $\widehat{\mu}_{T}$ and $\widehat{\gamma}_{T}$, where $\theta_{0}=(\alpha_{0},\mu_{0})\in(0,\infty)\times\mathbb{R}$ and $\gamma_{0}:=\alpha_{0}\mu_{0}$.
Note that the asymptotic minimax lower bounds (\ref{lower_bound2}) and (\ref{lower_bound1}) imply that the asymptotically efficient rate of estimating $\mu$ is $T^{1-H}$ and it is faster than $\sqrt{T}$ which is the asymptotically efficient rate of estimating $\gamma$ when $H<1/2$.

In order to prove (\ref{lower_bound2}) and (\ref{lower_bound1}), we first prove the Local Asymptotic Normality~(LAN) properties for the ergodic fOU process under both parameterizations (\ref{parameterization2}) and (\ref{parameterization1}). 
Then we can derive (\ref{lower_bound2}) and (\ref{lower_bound1}) using the LAN properties and the H\'ajek-Le Cam asymptotic minimax theorem. 
Unfortunately, \cite{Chiba-2020} proved that the LAN property for the ergodic fOU process under the parameterization (\ref{parameterization2}) does not hold due to the singularity of the Fisher information matrix when $H\in(1/4,1/2)$ and the diagonal matrix $T^{-1/2}I_{2}$ is used as the rate matrix, where $I_{p}$ denotes the identity matrix of size $p$. 
As shown later in detail, this phenomenon is due to the fact that each component of the score function rescaled by the rate-matrix $T^{-1/2}I_{2}$ is asymptotically linearly dependent; similar phenomena appear in the literature when we jointly estimate the self-similarity index and the volatility parameter of self-similar Gaussian noises or the stable L\'evy process with symmetric jumps under high-frequency observations, see \cite{Brouste-Fukasawa-2018,Fukasawa-Takabatake-2019,Brouste-Masuda-2018} for details.
Therefore, similar to the previous studies~\cite{Brouste-Fukasawa-2018,Fukasawa-Takabatake-2019,Brouste-Masuda-2018}, we introduce a suitable class of ``non-diagonal'' rate matrices in order 
to prove the LAN property for the ergodic fOU process under the parameterization (\ref{parameterization2}), see Section~\ref{section_main_results} for details.

This paper is organized as follows. 
We summarize notation used in this paper in Section~\ref{section_notation} and preliminary results of the fOU process in Section~\ref{section_fOU}.
Our main results including the LAN properties for the ergodic fOU process under both parameterizations (\ref{parameterization2}) and (\ref{parameterization1}) and the asymptotic efficiencies of the MLEs of $(\alpha,\mu)$ and $(\alpha,\gamma)$ are given in Section~\ref{section_main_results}.  
We derive asymptotic minimax lower bounds including (\ref{lower_bound2}) and (\ref{lower_bound1}) using the LAN properties for the ergodic fOU process with suitably selected sequences of non-diagonal rate matrices and the H\'ajek-Le Cam asymptotic minimax theorem in Section~\ref{section_lower_bounds}. 
Proofs of our main results are given in Section~\ref{section_proof_main_results}.
\section{Notation}\label{section_notation}
Let $\theta_{0}=(\alpha_{0},\mu_{0})\in(0,\infty)\times\mathbb{R}$ and $\gamma_{0}:=\alpha_{0}\mu_{0}$.
Let $(\Omega,\mathcal{F},\mathbb{P})$ be a complete probability space where a two-sided fractional Brownian motion $(B_{t}^{H})_{t\in\mathbb{R}}$ is defined. 
Denote by $C[0,T]$ the set of all $\mathbb{R}$-valued continuous functions on $[0,T]$ and by $\mathcal{B}(C[0,T])$ the Borel $\sigma$-algebra on $C[0,T]$ induced by the topology of uniform convergence. 
As shown in Section~\ref{subsection_solution_fOU}, there exists a pathwise unique solution $(X_{t}^{\theta})_{t\in[0,\infty)}$ of the fOU process on $(\Omega,\mathcal{F},\mathbb{P})$ and the solution $(X_{t}^{\theta})_{t\in[0,T]}$ induces a distribution $P_{\theta}^{T}$ on the measurable space $(C[0,T],\mathcal{B}(C[0,T]))$.  
Here we summarize definitions of several functions and stochastic processes used in representations of the likelihood ratio of the distributions $(P_{\theta}^{T})_{\theta\in\mathbb{R}^{2}}$.
\begin{enumerate}
	\item Define the transformation $\beta[f] \colon C[0,T]\to C[0,T]$ by
\begin{equation*}
	\beta_{t}[f](x):=\frac{}{}\sigma^{-1}\bar{d}_{H}^{-1}t^{H-1/2}\int_{0}^{t}(t-s)^{-1/2-H}s^{1/2-H}f(x_{s})\,\mathrm{d}s
\end{equation*}
for $f\in C(\mathbb{R})$ and $x\in C[0,T]$, where 
\begin{equation*}
	\bar{d}_{H}:=\Gamma(1/2-H)\left(\frac{2H\Gamma(3/2-H)\Gamma(H+1/2)}{\Gamma(2-2H)}\right)^{1/2}.
\end{equation*}
Using this transformation, we set
\begin{equation*}
	\beta_{t}(\theta):=\beta_{t}[-\alpha(\mathrm{id}-\mu)](\mathbf{X}_{T}^{\theta_{0}}) 
\end{equation*}
for $\theta=(\alpha,\mu)$ and $\mathbf{X}_{T}^{\theta_{0}}:=(X_{t}^{\theta_{0}})_{t\in[0,T]}$, where $\mathrm{id}(\cdot)$ denotes the identity map from $\mathbb{R}$ to itself. 
	\item
	Define by
	\begin{equation}\label{definition_Y}
		Y_t^{\theta_{0}}:=\frac{1}{\sigma}\int_{0}^{t}\eta_H(t,s)\,\mathrm{d}X_s^{\theta_{0}} =\int_{0}^{t}\beta_{s}(\theta_{0})\,\mathrm{d}s+W_t,
	\end{equation}
	where
	\begin{equation*}
		\eta_H(t,s):=\bar{d}_{H}^{-1}s^{1/2-H}\int_{s}^{t}(u-s)^{-1/2-H}u^{H-1/2}\,\mathrm{d}u,\ \ W_t:=\int_{0}^{t}\eta_H(t,s)\,\mathrm{d}B_s^H.
	\end{equation*}
	The stochastic integrals are defined in $L^{2}(\mathbb{P})$-sense and the process $W=(W_{t})_{t\in[0,\infty)}$ is a Wiener process on $(\Omega,\mathcal{F},\mathbb{P})$, see \cite{Tudor-Viens-2007} for details. 
	\item 
	Define the filtration $(\mathcal{F}_{t})_{t\in[0,\infty)}$ by $\mathcal{F}_{t}:=\mathcal{F}_{t}^{W}\vee\mathcal{N}$ for $t\in[0,\infty)$, where $(\mathcal{F}_{t}^{W})_{t\in[0,\infty)}$ denotes the natural filtration of $W$ and $\mathcal{N}$ denotes the family of all null sets of $\Omega$ with respect to $\mathbb{P}$.
	Then we define the sequences of the square integrable $(\mathcal{F}_{t})_{t\in[0,\infty)}$-martingales $(M_{t}^{T})_{t\in[0,\infty)}$ and $(N_{t}^{T})_{t\in[0,\infty)}$ by
	\begin{equation*}
		M_{t}^{T}:=\frac{1}{\sqrt{T}}\int_{0}^{t\wedge T}\beta_{s}[\mathrm{id}-\mu_{0}](\mathbf{X}_{T}^{\theta_{0}})\,\mathrm{d}W_{s},\ \ 
		N_{t}^{T}:=T^{H-1}\int_{0}^{t\wedge T}\beta_{s}[1]\,\mathrm{d}W_{s}.
	\end{equation*}
	Denote by $\langle X,Y\rangle$ the quadratic covariation of two continuous semi-martingales $X$ and $Y$ and set $\langle X\rangle:=\langle X,X\rangle$ for notational simplicity.
\end{enumerate}
Finally, $\stackrel{\mathbb{P}}{\to}$ and $\stackrel{\mathcal{L}}{\to}$ denote the convergence in probability and in law under the probability measure $\mathbb{P}$ respectively, and $a\lesssim b$ means that there exists a universal constant $C$ such that $a\leq Cb$.
\section{Preliminary Results of Fractional Ornstein-Uhlenbeck Process}\label{section_fOU}
\subsection{Stationary Solution of Fractional Ornstein-Uhlenbeck Process}\label{subsection_solution_fOU}
It is well-known that the SDE~(\ref{definition_SDE}) with the drift function defined by (\ref{parameterization1}) has the pathwise unique solution
\begin{equation*}
	X^{\theta}_{t}
	=e^{-\alpha t}x_{0} +(1-e^{-\alpha t})\mu +\sigma\int_{0}^{t}e^{-\alpha(t-s)}\,\mathrm{d}B_{s}^{H},\ \ t\in[0,\infty),
\end{equation*}
and the distance between the solution $X^{\theta}_{t}$ and the random variable $\bar{X}^{\theta}_{t}$ defined by
\begin{equation*}
	\bar{X}^{\theta}_{t} := \mu +\sigma\int_{-\infty}^{t} e^{-\alpha(t-s)}\,\mathrm{d}B_{s}^{H},\ \ t\in\mathbb{R},
\end{equation*}
decreases exponentially as $t\to\infty$ for any initial value $x_{0}$ when $\alpha>0$.
The stochastic process $\bar{X}^{\theta}=(\bar{X}^{\theta}_{t})_{t\in\mathbb{R}}$, defined on the probability space $(\Omega,\mathcal{F},\mathbb{P})$, is called the stationary solution of the fOU process.
It is also well-known that $\bar{X}^{\theta}$ is a stationary Gaussian process with mean $\mu$ and covariance function $c(t)$ given by
\begin{equation*}
	c(t):=\mathrm{Cov}[\bar{X}^{\theta}_{t},\bar{X}^{\theta}_{0}] 
	= \sigma^{2}\frac{\Gamma(2H+1)\sin(\pi H)}{2\pi}\int_{-\infty}^{\infty}e^{\sqrt{-1}tx}\frac{|x|^{1-2H}}{\alpha^{2}+x^{2}}\,\mathrm{d}x,\ \ t\in\mathbb{R},
\end{equation*}
and $c(t)=O(|t|^{2H-2})$ as $|t|\to\infty$. See \cite{Cheridito-Kawaguchi-Maejima-2003} for details.
\subsection{Representation of Likelihood Ratio}
As shown in \cite{D.Nualart-Ouknine-2002,Tudor-Viens-2007}, the family of distributions $(P_{\theta}^T)_{\theta\in\mathbb{R}^{2}}$ is mutually absolute continuous and its likelihood ratio has the following representation:
\begin{equation}\label{LR_formula1}
	\frac{\mathrm{d}P_{\theta}^{T}}{\mathrm{d}P_{\theta_{0}}^{T}}(\mathbf{X}_{T}^{\theta_{0}}) 
	=\exp\left(\int_{0}^{T}(\beta_{t}(\theta)-\beta_{t}(\theta_{0}))\,\mathrm{d}W_{t}-\frac{1}{2}\int_{0}^{T}(\beta_{t}(\theta)-\beta_{t}(\theta_{0}))^2\,\mathrm{d}t\right),
\end{equation}
which can be rewritten as
\begin{equation}\label{LR_formula2}
	\frac{\mathrm{d}P_{\theta}^{T}}{\mathrm{d}P_{\theta_{0}}^{T}}(\mathbf{X}_{T}^{\theta_{0}}) 
	=\exp\left(\int_{0}^{T}(\beta_{t}(\theta)-\beta_{t}(\theta_{0}))\,\mathrm{d}Y_t^{\theta_{0}}-\frac{1}{2}\int_{0}^{T}(\beta_{t}(\theta)^2-\beta_{t}(\theta_{0})^2)\,\mathrm{d}t\right)
\end{equation}
using the relation~(\ref{definition_Y}).
\subsection{Maximum Likelihood Estimators}
	We rewrite
	\begin{equation}\label{representation_beta_g}
		\beta_{t}(\theta)=\langle g(\theta),(-\beta_{t}[\mathrm{id}](\mathbf{X}_{T}^{\theta_{0}}),\beta_{t}[1])\rangle_{\mathbb{R}^2}
	\end{equation}
	using the function $g(\theta):=(\alpha,\alpha\mu)$ for $\theta=(\alpha,\mu)$. 
	Using (\ref{LR_formula2}) and (\ref{representation_beta_g}), the log-likelihood function $\ell_{T}(\theta)$ can be written as
	\begin{equation*}
		\ell_{T}(\theta):=\log{\frac{\mathrm{d}P_{\theta}^{T}}{\mathrm{d}P_{0}^{T}}(\mathbf{X}_{T}^{\theta_{0}})}
		=\left\langle g(\theta),\zeta_{T}(\theta_{0})\right\rangle_{\mathbb{R}^2}
		-\frac{1}{2}\left\langle g(\theta),\Gamma_{T}(\theta_{0})g(\theta)\right\rangle_{\mathbb{R}^2},
	\end{equation*}
	where 
	\begin{equation*}
		\zeta_{T}(\theta_{0}):=\int_{0}^{T}
		\begin{pmatrix}
		 -\beta_{t}[\mathrm{id}](\mathbf{X}_{T}^{\theta_{0}})\\
	 	 \beta_{t}[1]
	    \end{pmatrix}
		\,\mathrm{d}Y_t^{\theta_{0}},\ \ 
		\Gamma_{T}(\theta_{0})
		:=
		\int_{0}^{T}
		\begin{pmatrix}
	 	 -\beta_{t}[\mathrm{id}](\mathbf{X}_{T}^{\theta_{0}})\\
	 	 \beta_{t}[1]
	 	\end{pmatrix}^{\otimes 2}
		\,\mathrm{d}t.
	\end{equation*}
	Note that the maximization of $(\mathrm{d}P_{\theta}^{T}/\mathrm{d}P_{\theta_{0}}^{T})(\mathbf{X}_{T}^{\theta_{0}})$ with respect to $\theta$ is equivalent to that of $\ell_{T}(\theta)$ with respect to $\theta$.
	Using the Cauchy-Schwartz inequality, we have
	\begin{equation*}
		\det[\Gamma_{T}(\theta_{0})]>0\ \ \mbox{$\mathbb{P}$-a.s.}
	\end{equation*} 
	so that $g(\theta)\mapsto\ell_{T}(\theta)$ is concave on $\mathbb{R}^{2}$ $\mathbb{P}$-a.s. and the MLEs $\widehat{g(\theta)}_{T}$ and $\widehat{\theta}_{T}$ of the parameters $g(\theta)$ and $\theta$ respectively are 
	given by
	\begin{equation*}
		\widehat{g(\theta)}_{T}
		=\Gamma_{T}(\theta_{0})^{-1}\int_{0}^{T}
		\begin{pmatrix}
			-\beta_{t}[\mathrm{id}](\mathbf{X}_{T}^{\theta_{0}})\\
			\beta_{t}[1]
	  \end{pmatrix}\,\mathrm{d}Y_t^{\theta_{0}}\ \ \mbox{$\mathbb{P}$-a.s.}
	\end{equation*}
	and
	\begin{equation*}
		\widehat{\theta}_{T} = 
		\begin{cases}
			g^{-1}(\widehat{g(\theta)}_{T}) & \mathrm{if}\quad \widehat{g(\theta)}_{T} \in (\mathbb{R}\setminus\{0\})\times\mathbb{R},\\ 
			0 & otherwise,
		\end{cases}
	\end{equation*}
	where $g^{-1}$ is the inverse function of $g$ defined on $(\mathbb{R}\setminus\{0\})\times\mathbb{R}$, that is, $g^{-1}((\alpha,\kappa))=(\alpha,\kappa/\alpha)$. 
	\begin{remark}
	The MLE $\widehat{\theta}_{T}$ can be defined if the first component of $\widehat{g(\theta)}_{T}$ is not zero. Of course, we can show that the probability of the set where the MLE $\widehat{\theta}_{T}$ can not be defined decreases quickly as $T\to\infty$, see Lemma \ref{0928lemma1} for details.
	\end{remark}

	\begin{remark}[\textit{Asymptotic normality of $\widehat{g(\theta)}_{T}$ with a degenerate variance-covariance matrix}]\label{remark_degenerate_variance}
		As shown in \cite{Chiba-2020,Tanaka-Xiao-Yu-2019}, 
		even though each component of $\sqrt{T}(\widehat{g(\theta)}_{T}-g(\theta_{0}))$ converges to a non-degenerate centered normal distribution, $\sqrt{T}(\widehat{g(\theta)}_{T}-g(\theta_{0}))$ converges to a ``degenerate'' centered bivariate normal distribution.
		Indeed, using (\ref{definition_Y}), the MLE $\widehat{g(\theta)}_{T}$ can be rewritten as
		\begin{equation}\label{theta_tilde_representation}
			\widehat{g(\theta)}_{T}
			=g(\theta_{0})+\Gamma_{T}(\theta_{0})^{-1}\int_{0}^{T}
			\begin{pmatrix}
			 -\beta_{t}[\mathrm{id}](\mathbf{X}_{T}^{\theta_{0}})\\
			 \beta_{t}[1]
			\end{pmatrix}\,\mathrm{d}W_{t}.
		\end{equation}		
		Then we can show that
		\begin{equation}\label{error_g_theta_degenerate}
			\sqrt{T}(\widehat{g(\theta)}_{T}-g(\theta_{0}))
			=\frac{1}{\langle M^{T}\rangle_{T}+o_{\mathbb{P}}(1)}\left(
			M_{T}^{T}
			\begin{pmatrix}
			 -1\\
			 -\mu_{0}
			\end{pmatrix}
			+o_{\mathbb{P}}(1)\right)
			\stackrel{\mathcal{L}}{\to}\mathcal{N}\left(0,2\alpha_{0}
			\begin{pmatrix}
			 1\\
			 \mu_{0}
			 \end{pmatrix}^{\otimes 2}
			\right)
		\end{equation}
		as $T\to\infty$; we omit the details of the proof since (\ref{error_g_theta_degenerate}) can be proved in the similar way to the proof of Theorem~\ref{theorem_asymptotic_normality}.
		As seen in (\ref{error_g_theta_degenerate}), the variance-covariance matrix of $\sqrt{T}(\widehat{g(\theta)}_{T}-g(\theta_{0}))$ is asymptotically degenerate due to the asymptotically linear dependence of each component of $\sqrt{T}(\widehat{g(\theta)}_{T}-g(\theta_{0}))$.
	\end{remark}
\section{Main Results}\label{section_main_results}
Before stating our main results, we introduce a class of sequences of non-diagonal rate matrices in the similar way to \cite{Brouste-Fukasawa-2018} which plays an essential role to prove an asymptotic normality of the MLE with a non-degenerate asymptotic variance-covariance matrix, see also Remark~\ref{remark_degenerate_variance}, and the LAN property for the ergodic fOU process under the parametrization~(\ref{parameterization2}). 
\begin{assumption}\label{assumption_rate_matrix}
	Assume a sequence of matrices $(\varphi_{T}(\theta_{0}))_{T\in(0,\infty)}$ and a matrix $\overline{\varphi}(\theta_{0})$ of the form
	\begin{equation*}
		\varphi_{T}(\theta_{0}):=
		\begin{pmatrix}
			\varphi_{T}^{11}(\theta_{0})&\varphi_{T}^{12}(\theta_{0})\\
			\varphi_{T}^{21}(\theta_{0})&\varphi_{T}^{22}(\theta_{0})
		\end{pmatrix},\ \
		\overline{\varphi}(\theta_{0}):=
		\begin{pmatrix}
			\overline{\varphi}_{11}(\theta_{0})&\overline{\varphi}_{12}(\theta_{0})\\
			\overline{\varphi}_{21}(\theta_{0})&\overline{\varphi}_{22}(\theta_{0})
		\end{pmatrix}
	\end{equation*}
	satisfy the following conditions:
	\begin{enumerate}
		\item $\sqrt{T}\varphi_{T}^{11}(\theta_{0})\rightarrow\overline\varphi_{11}(\theta_{0})$ as $T\to\infty$,
		\item $\sqrt{T}\varphi_{T}^{12}(\theta_{0})\rightarrow\overline\varphi_{12}(\theta_{0})$ as $T\to\infty$,
		\item $s_{T}^{21}(\theta_{0}):=T^{1-H}(\mu\varphi_{T}^{11}(\theta_{0})-\varphi_{T}^{21}(\theta_{0})) \rightarrow\overline\varphi_{21}(\theta_{0})$ as $T\to\infty$,
		\item $s_{T}^{22}(\theta_{0}):=T^{1-H}(\mu\varphi_{T}^{12}(\theta_{0})-\varphi_{T}^{22}(\theta_{0})) \rightarrow\overline\varphi_{22}(\theta_{0})$ as $T\to\infty$,
		\item $\varphi_{T}^{11}(\theta_{0})\varphi_{T}^{22}(\theta_{0})-\varphi_{T}^{12}(\theta_{0})\varphi_{T}^{21}(\theta_{0})\neq 0$ for each $T>0$,
		\item $\overline\varphi_{11}(\theta_{0})\overline\varphi_{22}(\theta_{0})-\overline\varphi_{12}(\theta_{0})\overline\varphi_{21}(\theta_{0})\neq 0$.
	\end{enumerate}
\end{assumption}
\begin{remark}[\textit{Examples of $\varphi_{T}(\theta_{0})$}]
	For example, we can take 
	\begin{equation}\label{example_rate_matrix1}
		\varphi_{T}^{11}(\theta_{0})=\frac{1}{\sqrt{T}},\ \ \varphi_{T}^{21}(\theta_{0})=\frac{\mu_{0}}{\sqrt{T}},\ \ \varphi_{T}^{12}(\theta_{0})=0,\ \ \varphi_{T}^{22}(\theta_{0})=-\frac{1}{T^{1-H}},
	\end{equation}
	which give $\overline\varphi_{11}(\theta_{0})=1$, $\overline\varphi_{21}(\theta_{0})=0$, $\overline\varphi_{12}(\theta_{0})=0$ and $\overline\varphi_{22}(\theta_{0})=1$, or
	\begin{equation}\label{example_rate_matrix2}
		\varphi_{T}^{11}(\theta_{0})=\frac{1}{\mu_{0}T^{1-H}},\ \ \varphi_{T}^{21}(\theta_{0})=0,\ \ \varphi_{T}^{12}(\theta_{0})=\frac{1}{\sqrt{T}},\ \ \varphi_{T}^{22}(\theta_{0})=\frac{\mu_{0}}{\sqrt{T}},
	\end{equation}
	which give $\overline\varphi_{11}(\theta_{0})=0$, $\overline\varphi_{21}(\theta_{0})=1$, $\overline\varphi_{12}(\theta_{0})=1$ and $\overline\varphi_{22}(\theta_{0})=0$.
\end{remark}

\begin{remark}
	If a sequence of matrices $(\varphi_{T}(\theta_{0}))_{T\in(0,\infty)}$ satisfying Assumption~\ref{assumption_rate_matrix}, then
	\begin{equation}\label{property_varphi_tilde}
		\widetilde{\varphi}_T(\theta_{0}):=
		\begin{pmatrix}
			\sqrt{T}&\mu_{0}T^{1-H}\\
			0&-T^{1-H}
		\end{pmatrix}^\ast
		\varphi_{T}(\theta_{0})=
		\begin{pmatrix}
			\sqrt{T}\varphi_{T}^{11}(\theta_{0})&\sqrt{T}\varphi_{T}^{12}(\theta_{0})\\
			s_{T}^{21}(\theta_{0})&s_{T}^{22}(\theta_{0})
		\end{pmatrix}
		\stackrel{T\to\infty}{\rightarrow}\overline{\varphi}(\theta_{0}).
	\end{equation}
\end{remark}
Then we can prove the following asymptotic normality of the MLE $\widehat{g(\theta)}_{T}$.
\begin{theorem}\label{theorem_asymptotic_normality}
	For a sequence of matrices $(\varphi_{T}(\theta_{0}))_{T\in(0,\infty)}$ satisfying Assumption~$\ref{assumption_rate_matrix}$,
	\begin{equation*}
		\varphi_{T}(\theta_{0})^{-1}(\widehat{g(\theta)}_{T}-g(\theta_{0}))\stackrel{\mathcal{L}}{\to} \mathcal{N}(0,\mathcal{I}(\theta_{0})^{-1})\ \ \mbox{as $T\to\infty$,}
	\end{equation*}
	where $\mathcal{I}(\theta_{0})$ is the positive definite matrix defined by
	\begin{equation}\label{definition_Fisher_information}
		\mathcal{I}(\theta_{0}):=\overline{\varphi}(\theta_{0})^{\ast}\mathrm{diag}\left((2\alpha_{0})^{-1},(\sigma^{2}\lambda_{H})^{-1}\right)\overline{\varphi}(\theta_{0})
	\end{equation}
	with
	\begin{equation*}
		\lambda_{H}:=\frac{2H\Gamma(3-2H)\Gamma(H+1/2)}{\Gamma(3/2-H)}.
	\end{equation*}
\end{theorem}
We can also prove the following asymptotic normality of the MLE $\widehat{\theta}_{T}$.
\begin{theorem}\label{corollary_asymptotic_normality}
	Consider a sequence of matrices $(\varphi_{T}(\theta_{0}))_{T\in(0,\infty)}$ satisfying Assumption~$\ref{assumption_rate_matrix}$. 
	Set 	
	\begin{equation}\label{definition_Psi_theta}
		\Psi_{T}(\theta_{0}):=J_{g^{-1}}(g(\theta_{0}))\varphi_{T}(\theta_{0})=J_{g}(\theta_{0})^{-1}\varphi_{T}(\theta_{0}),
	\end{equation}
	where $J_{f}(x)$ denotes the Jacobian matrix of a function $f$ at the point $x$. Then 
	\begin{equation*}
		\Psi_{T}(\theta_{0})^{-1}(\widehat{\theta}_{T}-\theta_{0})\stackrel{\mathcal{L}}{\to} \mathcal{N}(0,\mathcal{I}(\theta_{0})^{-1})\ \ \mbox{as $T\to\infty$.}
	\end{equation*}
\end{theorem}
\begin{remark}\label{remark_Psi_diagonal}
	We can take some diagonal matrices as the rate matrix $\Psi_{T}(\theta_{0})$. For example, we can take
	\begin{equation}\label{diagonal_rate_matrix}
		\Psi_{T}(\theta_{0})=\mathrm{diag}(\sqrt{T},-\alpha_{0}T^{1-H})^{-1}
	\end{equation}
	when the elements of the matrix $\varphi_{T}(\theta_{0})$ are defined by (\ref{example_rate_matrix1}). Using the rate matrix $\Psi_{T}(\theta_{0})$ defined by (\ref{diagonal_rate_matrix}) and Theorem~\ref{corollary_asymptotic_normality}, we obtain
	\begin{equation*}
		\mathrm{diag}(\sqrt{T},-\alpha_{0}T^{1-H})(\widehat{\theta}_{T}-\theta_{0})\stackrel{\mathcal{L}}{\to}
		\mathcal{N}\left(0,\mathrm{diag}(2\alpha_{0},\sigma^{2}\lambda_{H})\right)\ \ \mbox{as $T\to\infty$},
	\end{equation*}
	which implies each component of the suitably rescaled estimation error of the MLE $\widehat{\theta}_{T}=(\widehat{\alpha}_{T},\widehat{\mu}_{T})$ is asymptotically normal and asymptotically independent.
	Then we can construct confidence intervals of the mean-reverting coefficient $\alpha$ and the mean-reverting level $\mu$ using the central limit theorems
	\begin{equation*}
		\sqrt{\frac{T}{2\alpha_{0}}}(\widehat{\alpha}_{T}-\alpha_{0})\stackrel{\mathcal{L}}{\to}\mathcal{N}(0,1),\ \ 
		\frac{\widehat{\alpha}_{T}T^{1-H}}{\sigma\sqrt{\lambda_{H}}}(\widehat{\mu}_{T}-\mu_{0})\stackrel{\mathcal{L}}{\to}\mathcal{N}(0,1)\ \ \mbox{as $T\to\infty$.}
	\end{equation*}
\end{remark}
Moreover, the LAN property for the ergodic fOU process under continuous observations holds for each parameterization~(\ref{parameterization2}) and (\ref{parameterization1}) as follows.
\begin{theorem}\label{theorem_LAN}
	Consider a sequence of matrices $(\varphi_{T}(\theta_{0}))_{T\in(0,\infty)}$ satisfying Assumption~$\ref{assumption_rate_matrix}$ and set the rate matrix $\Psi_{T}(\theta_{0})$ as $(\ref{definition_Psi_theta})$. Then the following LAN properties hold as $T\to\infty$:
	\begin{align}
		\label{LAN_parameterization2}
		&\frac{\mathrm{d}P_{g^{-1}(g(\theta_{0})+\varphi_{T}(\theta_{0})u)}^T}{\mathrm{d}P_{\theta_{0}}^T}(\mathbf{X}_{T}^{\theta_{0}})
		=\exp\left(\left\langle u,\Delta_{T}(\theta_{0})\right\rangle_{\mathbb{R}^{2}} -\frac{1}{2}\left\langle u,\mathcal{I}(\theta_{0})u\right\rangle_{\mathbb{R}^{2}}+o_{\mathbb{P}}(1)\right),\\
		\label{LAN_parameterization1}
		&\frac{\mathrm{d}P_{\theta_{0}+\Psi_{T}(\theta_{0})u}^T}{\mathrm{d}P_{\theta_{0}}^T}(\mathbf{X}_{T}^{\theta_{0}})
		=\exp\left(\left\langle u,\Delta_{T}(\theta_{0})\right\rangle_{\mathbb{R}^{2}} -\frac{1}{2}\left\langle u,\mathcal{I}(\theta_{0})u\right\rangle_{\mathbb{R}^{2}}+o_{\mathbb{P}}(1)\right)
	\end{align}
	for each $u\in\mathbb{R}^{2}$, where the positive definite matrix $\mathcal{I}(\theta_{0})$ is defined by $(\ref{definition_Fisher_information})$ and the random vector $\Delta_{T}(\theta_{0}):=-\widetilde{\varphi}_{T}(\theta_{0})^{\ast}(M_{T}^{T},N_{T}^{T})$ satisfies $\Delta_{T}(\theta_{0})\stackrel{\mathcal{L}}{\to}\mathcal{N}(0,\mathcal{I}(\theta_{0}))$ as $T\to\infty$. 
\end{theorem}
\begin{remark}
	(\ref{LAN_parameterization2}) means that the family of distributions $(P_{g^{-1}(\tilde{\theta})}^{T})_{\tilde{\theta}\in(0,\infty)\times\mathbb{R}^{2}}$, which are corresponding to those of the ergodic fOU process under the parametrization~(\ref{parameterization2}), enjoys the LAN property at the point $g(\theta_{0})=(\alpha_{0},\gamma_{0})\in(0,\infty)\times\mathbb{R}$ for the sequence of non-degenerate matrices $(\varphi_{T}(\theta_{0}))_{T\in(0,\infty)}$ and the positive definite matrix $\mathcal{I}(\theta_{0})$.
\end{remark}
\section{Asymptotic Efficiency of MLEs}\label{section_lower_bounds}
In this section, we derive asymptotically efficient rates and variances of estimators of $\alpha$, $\gamma$ and $\mu$ respectively based on the LAN property proved in Theorem~\ref{theorem_LAN}.
First of all, we recall the H\'ajek-Le Cam asymptotic minimax theorem which gives us asymptotic minimax lower bounds of estimators.
\begin{theorem}[The H\'ajek-Le Cam asymptotic minimax theorem~(\cite{Hajek1972,LeCam1972,Ibragimov-Hasminski-1981})]\label{Thm_MinMax}
	Let $\Theta\subset\mathbb{R}^d$ and a family of distributions $\{P_{\theta}^{T}\}_{\theta\in\Theta}$ on some measurable spaces $(\mathcal{X}_{T},\mathcal{A}_{T})$ satisfy the LAN property at $\theta_{0}\in\Theta$ as $T\to\infty$ for a sequence of regular matrices $(\varphi_{T}(\theta_{0}))_{T\in(0,\infty)}$ satisfying $\mathrm{Tr}[\varphi_{T}(\theta_{0})\varphi_{T}(\theta_{0})^{\ast}]\to 0$ and a positive definite matirx $\mathcal{I}(\theta_{0})$. Then we obtain
	\begin{equation*}
		\varliminf_{T\to\infty}\sup_{|\theta-\theta_{0}|<r}E_{\theta}^{T}\left[w\left(\varphi_{T}(\theta_{0})^{-1}(\widehat{\theta}_{T}-\theta)\right)\right]
		\geq\int_{\mathbb{R}^d}w\left(\mathcal{I}(\theta_{0})^{-\frac{1}{2}}z\right)\phi_{d}(z)\,\mathrm{d}z
	\end{equation*}
	for each $r>0$, any sequence of estimators $\widehat{\theta}_{T}$ and any symmetric, non-negative quasi-convex\footnote{
		A function $w:\mathbb{R}^d\to\mathbb{R}$ is quasi-convex if the set $\{x\in\mathbb{R}^d:w(x)\leq c\}$ is convex for each $c>0$.
	} function $w$ with
	\begin{equation*}
	\lim_{|z|\to\infty}e^{-\varepsilon|z|^2}w(z)=0
	\end{equation*}
	for all $\varepsilon>0$, where $\phi_{d}(\cdot)$ is the probability density function of the d-dimensional standard normal distribution.
\end{theorem}
\subsection{Asymptotically Efficient Rate and Variance of Estimating $\alpha$}
As the rate matrix $\varphi_{T}(\theta_{0})$, let us take the matrix whose elements are defined by (\ref{example_rate_matrix1}):
\begin{equation}\label{example_rate_matrix1_ex}
	\varphi_{T}(\theta_{0})=
	\begin{pmatrix}
		1&0\\
		\mu_{0}&-1
	\end{pmatrix}
	\mathrm{diag}(\sqrt{T},T^{1-H})^{-1}.
\end{equation}
It is worth mentioning that the rate matrix is non-diagonal and depends on the parameter $\mu_{0}$.
Using Theorems~\ref{theorem_LAN} and \ref{Thm_MinMax}, we obtain
\begin{equation*}
	\varliminf_{T\to\infty}\sup_{|\theta-\theta_{0}|<r}E_{\theta}^{T}\left[w\left(\varphi_{T}(\theta_{0})^{-1}(\widehat{g(\theta)}_{T}-g(\theta))\right)\right]\geq\int_{\mathbb{R}^2}w\left(
	\mathrm{diag}(2\alpha_{0},\sigma^{2}\lambda_{H})^{\frac{1}{2}}
	z\right)\phi_{2}(z)\,\mathrm{d}z
\end{equation*}
for each $r>0$, any sequence of estimators $\widehat{g(\theta)}_{T}=(\widehat{\alpha}_{T},\widehat{\gamma}_{T})$ and any loss function $w$ satisfying the conditions given in Theorem~\ref{Thm_MinMax}.
Since
\begin{equation*}
	\varphi_{T}(\theta_{0})^{-1}=\mathrm{diag}(\sqrt{T},T^{1-H})
	\begin{pmatrix}
		1&0 \\
		\mu_{0}&-1
	\end{pmatrix},
\end{equation*}
we obtain the asymptotic minimax lower bound
\begin{equation*}
	\varliminf_{T\to\infty}\sup_{|\theta-\theta_{0}|<r}E_{\theta}^{T}[T(\widehat{\alpha}_{T}-\alpha)^2]\geq 2\alpha_{0}
\end{equation*}
by taking $w(x,y)=x^2$.
\subsection{Asymptotically Efficient Rate and Variance of Estimating $\gamma$}
As the rate matrix $\varphi_{T}(\theta_{0})$, let us take the matrix whose elements are defined by (\ref{example_rate_matrix2}):
\begin{equation*}
	\varphi_{T}(\theta_{0})=
	\begin{pmatrix}
		\mu_{0}^{-1}&1 \\
		0&\mu_{0}
	\end{pmatrix}
	\mathrm{diag}(T^{1-H},\sqrt{T})^{-1}.
\end{equation*}
It is worth repeating that the rate matrix is non-diagonal and depends on the parameter $\mu_{0}$.
Using Theorems~\ref{theorem_LAN} and \ref{Thm_MinMax}, we obtain
\begin{equation*}
	\varliminf_{T\to\infty}\sup_{|\theta-\theta_{0}|<r}E_{\theta}^{T}\left[w\left(\varphi_{T}(\theta_{0})^{-1}(\widehat{g(\theta)}_{T}-g(\theta))\right)\right]
	\geq\int_{\mathbb{R}^2}w\left(
	\mathrm{diag}(\sigma^{2}\lambda_{H},2\alpha_{0})^{\frac{1}{2}}
	z\right)\phi_{2}(z)\,\mathrm{d}z
\end{equation*}
for each $r>0$, any sequence of estimators $\widehat{g(\theta)}_{T}=(\widehat{\alpha}_{T},\widehat{\gamma}_{T})$ and any loss function $w$ satisfying the conditions given in Theorem~\ref{Thm_MinMax} because we have
\begin{equation*}
	\mathcal{I}(\theta_{0})^{-1}=
	\begin{pmatrix}
		0&1\\
		1&0
	\end{pmatrix}
	\mathrm{diag}(2\alpha_{0},\sigma^{2}\lambda_{H})
	\begin{pmatrix}
		0&1\\
		1&0
	\end{pmatrix}
	=
	\begin{pmatrix}
		\sigma^{2}\lambda_{H}&0\\
		0&2\alpha_{0}
	\end{pmatrix}.
\end{equation*}
Since
\begin{equation*}
	\varphi_{T}(\theta_{0})^{-1}
	=\mathrm{diag}(T^{1-H},\sqrt{T})
	\begin{pmatrix}
		\mu_{0}&-1 \\
		0&\mu_{0}^{-1}
	\end{pmatrix},
\end{equation*}
we obtain the asymptotic minimax lower bound
\begin{equation*}
	\varliminf_{T\to\infty}\sup_{|\theta-\theta_{0}|<r}E_{\theta}^{T}[T(\widehat{\gamma}_{T}-\gamma)^2]\geq 2\alpha_{0}\mu_{0}^{2}=\frac{2\gamma_{0}^{2}}{\alpha_{0}}
\end{equation*}
by taking $w(x,y)=y^2$.
\subsection{Asymptotically Efficient Rate and Variance of Estimating $\mu$}
As the rate matrix $\varphi_{T}(\theta_{0})$, let us take the same one as (\ref{example_rate_matrix1_ex}) which gives
\begin{equation*}
	\Psi_{T}(\theta_{0})=\mathrm{diag}(\sqrt{T},-\alpha_{0}T^{1-H})^{-1},
\end{equation*}
see Remark~\ref{remark_Psi_diagonal}. Using Theorems~\ref{theorem_LAN} and \ref{Thm_MinMax}, we obtain
\begin{equation*}
	\varliminf_{T\to\infty}\sup_{|\theta-\theta_{0}|<r}E_{\theta}^{T}\left[w\left(\Psi_{T}(\theta_{0})^{-1}(\widehat{\theta}_{T}-\theta)\right)\right]
	\geq\int_{\mathbb{R}^2}w\left(
	\mathrm{diag}(2\alpha_{0},\sigma^{2}\lambda_{H})^{\frac{1}{2}}
	z\right)\phi_{2}(z)\,\mathrm{d}z
\end{equation*}
for each $r>0$, any sequence of estimators $\widehat{\theta}_{T}=(\widehat{\alpha}_{T},\widehat{\mu}_{T})$ and any loss function $w$ satisfying the conditions given in Theorem~\ref{Thm_MinMax}.
Since
\begin{equation*}
	\Psi_{T}(\theta_{0})^{-1}=\mathrm{diag}(\sqrt{T},-\alpha_{0}T^{1-H}),
\end{equation*}
we obtain the asymptotic minimax lower bound 
\begin{equation*}
	\varliminf_{T\to\infty}\sup_{|\theta-\theta_{0}|<r}E_{\theta}^{T}[T^{2(1-H)}(\widehat{\mu}_{T}-\mu)^2]\geq\frac{\sigma^{2}\lambda_{H}}{\alpha_{0}^{2}}
\end{equation*}
by taking $w(x,y)=y^{2}$.
\section{Proof of Main Results}\label{section_proof_main_results}

\subsection{Preliminary Results}
The following lemma is used in the proofs of the main results; its proof will be given in Section~\ref{section_proof_preliminary_results}.
\begin{proposition}\label{fOU_observed_Fisher}\label{proposition_Fisher_mean-reverting}
	For any $p>1$ and $H\in(0,1/2)$, there exists a constant $\kappa(H)\in(0,1)$ such that
\begin{align}
	\mathbb{E}\left[\left|\langle M^{T}\rangle_{T}-\frac{1}{2\alpha_{0}}\right|^p\right]&\lesssim T^{-\kappa(H)p},\label{0928ineq1}\\
	\mathbb{E}\left[\left|\langle M^{T},N^{T}\rangle_{T}\right|^p \right]&\lesssim T^{-\kappa(H)p}.\label{0928ineq2}
\end{align}
\end{proposition}

\begin{remark}
For any $T>0$, we have
\begin{equation}\label{N_quad_var}
	\langle N^{T}\rangle_{T} = d_{H}^{-2}\sigma^{-2}(2-2H)^{-1}\left(\int_{0}^{1} (1-s)^{-\frac{1}{2}-H}s^{\frac{1}{2}-H}\,\mathrm{d}s\right)^{2}
	=(\sigma^{2}\lambda_{H})^{-1}.
\end{equation} 
\end{remark}
\subsection{Proof of Theorem~\ref{theorem_asymptotic_normality}}
	Using (\ref{theta_tilde_representation}), (\ref{property_varphi_tilde}) and the equality
	\begin{equation*}
		\begin{pmatrix}
			-\beta_{t}[\mathrm{id}](\mathbf{X}_{T}^{\theta_{0}})\\
			\beta_{t}[1]
		\end{pmatrix}
		=-
		\begin{pmatrix}
			\sqrt{T}&\mu_{0}T^{1-H}\\
			0&-T^{1-H}
		\end{pmatrix}
		\begin{pmatrix}
			T^{-1/2}\beta_{t}[\mathrm{id}-\mu_{0}](\mathbf{X}_{T}^{\theta_{0}})\\
			T^{H-1}\beta_{t}[1]
		\end{pmatrix},
	\end{equation*}
	we can rewrite
	\begin{equation}\label{rescaled_est_err}
		\varphi_{T}(\theta_{0})^{-1}(\widehat{g(\theta)}_{T}-g(\theta_{0}))
		=-\widetilde{\Gamma}_{T}
		(\theta_{0})^{-1}\widetilde{\varphi}_T(\theta_{0})^\ast
		\begin{pmatrix}
			M_{T}^{T}\\
			N_{T}^{T}
		\end{pmatrix},
	\end{equation}
	where 
	\begin{equation*}
		\widetilde{\Gamma}_{T}(\theta):=\varphi_{T}(\theta)^\ast\Gamma_{T}(\theta)\varphi_{T}(\theta)
		=\widetilde{\varphi}_{T}(\theta_{0})^\ast
		\begin{pmatrix}
			\langle M^{T}\rangle_{T}&\mathrm{sym.}\\
			\langle M^{T},N^{T}\rangle_{T}&\langle N^{T}\rangle_{T}
		\end{pmatrix}
		\widetilde{\varphi}_T(\theta_{0}).
	\end{equation*}
	Using Proposition~\ref{proposition_Fisher_mean-reverting} and (\ref{N_quad_var}), we can show
	\begin{equation*}
		\begin{pmatrix}
			\langle M^{T}\rangle_{T}&\mathrm{sym.}\\
			\langle M^{T},N^{T}\rangle_{T}&\langle N^{T}\rangle_{T}
		\end{pmatrix}
		\stackrel{\mathbb{P}}{\to}\mathrm{diag}\left((2\alpha_{0})^{-1},(\sigma^{2}\lambda_{H})^{-1}\right)\ \ \mbox{as $T\to\infty$}
	\end{equation*}
	so that, using (\ref{property_varphi_tilde}), the martingale CLT and Slutsky's theorem, we obtain
	\begin{equation}\label{key_CLT}
		-\widetilde{\varphi}_{T}(\theta_{0})^{\ast}
		\begin{pmatrix}
			M_{T}^{T}\\
			N_{T}^{T}
		\end{pmatrix}
		\stackrel{\mathcal{L}}{\rightarrow}\mathcal{N}\left(\mathrm{0},\mathcal{I}(\theta_{0})\right),\ \ \widetilde{\Gamma}_{T}(\theta_{0})\stackrel{\mathbb{P}}{\to}\mathcal{I}(\theta_{0})\ \ \mbox{as $T\to\infty$}.
	\end{equation}
	Therefore the conclusion follows from (\ref{key_CLT}) using Slutsky's theorem again.	
\subsection{Proof of Theorem~\ref{corollary_asymptotic_normality}} 
For each $\epsilon>0$, we set
\[
A_{T,\epsilon} := \{|\widehat{g(\theta)}_{T} - g(\theta_{0}) | \geq \epsilon\}.
\]
We first prove the following lemma related to the deviation probability for $\widehat{g(\theta)}_{T}$.
\begin{lemma} \label{0928lemma1}
For any $q>0$ and $\epsilon>0$, $\lim_{T\to\infty} T^{q} \mathbb{P}\left[ A_{T,\epsilon} \right]  = 0.$
\end{lemma}
\begin{proof}
First note that some straightforward calculations show
\[
\det\Gamma_{T}(\theta_{0}) = T^{3-2H}(\langle M^{T}\rangle_{T}\langle N^{T}\rangle_{T} - \langle M^{T}, N^{T}\rangle_{T}^{2})
\]
and 
\begin{align*}
	\widehat{g(\theta)}_{T} - g(\theta_{0}) &= \Gamma_{T}(\theta_{0})^{-1}\int_{0}^{T} 
			\begin{pmatrix}
			 -\beta_{t}[\mathrm{id}](\mathbf{X}_{T}^{\theta_{0}})\\
			 \beta_{t}[1]
			\end{pmatrix}\,\mathrm{d}W_{t} \\ 
			& = \frac{T^{-\frac{1}{2}}}{\langle M^{T}\rangle_{T}\langle N^{T}\rangle_{T} - \langle M^{T}, N^{T}\rangle_{T}^{2}}Z_{T},
\end{align*}
where 
\begin{equation*}
	Z_{T} :=  \begin{pmatrix} -\langle N^{T}\rangle_{T}M_{T}^{T} + \langle M^{T},N^{T}\rangle_{T}N^{T}_{T} \\ -T^{H-\frac{1}{2}}(\langle M^{T},N^{T}\rangle_{T}M_{T}^{T} - \langle M^{T}\rangle_{T}N_{T}^{T}) + \mu_{0}(\langle M^{T},N^{T}\rangle_{T}N_{T}^{T} - \langle N^{T}\rangle_{T}M_{T}^{T}) \end{pmatrix}.
\end{equation*}
Then, for any $p>1$, we have
\begin{align*}
\mathbb{P}\left[ A_{T,\epsilon} \right] &\leq \mathbb{P}\left[ \left|\frac{\det\Gamma_{T}(\theta_{0})}{T^{3-2H}} - \frac{\langle N^{T}\rangle_{T}}{2\alpha_{0}}\right| \geq \frac{\langle N^{T}\rangle_{T}}{4\alpha_{0}} \right] + \mathbb{P}\left[ \left(\frac{\langle N^{T}\rangle_{T}}{4\alpha_{0}}\right)^{-1} T^{-1/2}\left|Z_{T}\right| \geq \epsilon  \right]\\
&\lesssim \mathbb{E}\left[ \left|\frac{\det\Gamma_{T}(\theta_{0})}{T^{3-2H}} - \frac{\langle N^{T}\rangle_{T}}{2\alpha_{0}}\right|^{p} \right] + \epsilon^{-p}T^{-\frac{p}{2}} \mathbb{E}\left[|Z_{T}|^{p}\right].  
\end{align*}
Since $p>1$ is arbitrary, the conclusion follows from Proposition \ref{proposition_Fisher_mean-reverting}.
\end{proof}

Now we turn to the proof of Theorem~\ref{corollary_asymptotic_normality}. On the set $\Omega \setminus A_{T,\epsilon}$, the first component of $\widehat{g(\theta)}_{T}$ is positive and the MLE is well-defined for sufficiently small $\epsilon>0$. We decompose 
\begin{align*}
\Psi_{T}(\theta_{0})^{-1}(\widehat{\theta}_{T}-\theta_{0}) &= \Psi_{T}(\theta_{0})^{-1}(g^{-1}(\widehat{g(\theta)}_{T})-g^{-1}(g(\theta_{0}))) \mathbf{1}_{\Omega\setminus A_{T,\epsilon}}\\ 
& \quad + \Psi_{T}(\theta_{0})^{-1}(-\theta_{0}) \mathbf{1}_{A_{T,\epsilon}}.
\end{align*}
The first term converges to the normal distribution $\mathcal{N}(0,\mathcal{I}(\theta_{0})^{-1})$ as $T\to\infty$ by the delta method and Slutsky's theorem.
Moreover, the second term tends to 0 as $T\to\infty$ in probability, thanks to Lemma~\ref{0928lemma1}. 
This completes the proof.
\color{black}

\subsection{Proof of (\ref{LAN_parameterization2}) in Theorem~\ref{theorem_LAN}} 
Note that, using (\ref{LR_formula1}) and (\ref{representation_beta_g}), we can rewrite
\begin{equation}\label{likelihood_ratio_W}
	\log{\frac{\mathrm{d}P_{\theta}^T}{\mathrm{d}P_{\theta_{0}}^T}}(\mathbf{X}_{T}^{\theta_{0}})
	=\langle g(\theta)-g(\theta_{0}),\overline{\zeta}_{T}(\theta_{0})\rangle_{\mathbb{R}^2}
		-\frac{1}{2}\left\langle g(\theta)-g(\theta_{0}),\Gamma_{T}(\theta_{0})[g(\theta)-g(\theta_{0})]\right\rangle_{\mathbb{R}^2},
\end{equation}
where
\begin{equation*}
	\overline{\zeta}_{T}(\theta_{0}):=\int_{0}^{T}
	\begin{pmatrix}
	  -\beta_{t}[\mathrm{id}](\mathbf{X}_{T}^{\theta_{0}})\\
	  \beta_{t}[1]
	\end{pmatrix}
	\,\mathrm{d}W_t.
\end{equation*}
Let us take a sufficiently large $T>0$ such that $u\in\mathbb{R}^{2}$ satisfies $g(\theta_{0})+\varphi_{T}(\theta_{0})u\in(0,\infty)\times\mathbb{R}$. 
Then, using (\ref{likelihood_ratio_W}), we obtain
\begin{equation*}
	\log{\frac{\mathrm{d}P_{g^{-1}(g(\theta_{0})+\varphi_{T}(\theta_{0})u)}^T}{\mathrm{d}P_{\theta_{0}}^T}}(\mathbf{X}_{T}^{\theta_{0}})
	=\left\langle u,\Delta_{T}(\theta_{0})\right\rangle_{\mathbb{R}^2}
	-\frac{1}{2}\left\langle u,\mathcal{I}(\theta_{0})u\right\rangle_{\mathbb{R}^2}
	+r_{T}^{(0)}(u;\theta_{0}),
\end{equation*}
where
\begin{equation}\label{definition_r0}
	r_{T}^{(0)}(u;\theta_{0}):=-\frac{1}{2}\langle u,[\widetilde{\Gamma}_{T}(\theta_{0})
	-\mathcal{I}(\theta_{0})]u\rangle_{\mathbb{R}^2}.
\end{equation}
Therefore the conclusion follows from (\ref{key_CLT}).
\subsection{Proof of (\ref{LAN_parameterization1}) in Theorem~\ref{theorem_LAN}}
Set
\begin{align*}
	&R_{T}(u;\theta_{0}):=
	\int_{0}^{1}[J_{g}(\theta_{0}+\epsilon\Psi_{T}(\theta_{0})u)-J_{g}(\theta_{0})]
	\Psi_{T}(\theta_{0})u\,\mathrm{d}\epsilon,\\ 
	&\widetilde{R}_{T}(u;\theta_{0}):=\varphi_{T}(\theta_{0})^{-1}R_{T}(u;\theta_{0}).
\end{align*}
Using (\ref{likelihood_ratio_W}) and Taylor's theorem, we can show
\begin{align*}
	&\log{\frac{\mathrm{d}P_{\theta_{0}+\Psi_{T}(\theta_{0})u}^T}{\mathrm{d}P_{\theta_{0}}^T}}(\mathbf{X}_{T}^{\theta_{0}})\\
	&=\langle J_{g}(\theta_{0})\Psi_{T}(\theta_{0})u,\zeta_{T}(\theta_{0})\rangle_{\mathbb{R}^2}
	-\frac{1}{2}\langle J_{g}(\theta_{0})\Psi_{T}(\theta_{0})u,\Gamma_{T}(\theta_{0})J_{g}(\theta_{0})\Psi_{T}(\theta_{0})u\rangle_{\mathbb{R}^2} +r_{T}(u;\theta_{0})\\
	&=\left\langle u,\Delta_{T}(\theta_{0})\right\rangle_{\mathbb{R}^2}
	-\frac{1}{2}\left\langle u,\mathcal{I}(\theta_{0})u\right\rangle_{\mathbb{R}^2} +r_{T}^{(0)}(u;\theta_{0}) +r_{T}(u;\theta_{0}),
\end{align*}
where $r_{T}^{(0)}(u;\theta_{0})$ and $r_{T}(u;\theta_{0}):=r_{T}^{(1)}(u;\theta_{0})+r_{T}^{(2)}(u;\theta_{0})+r_{T}^{(3)}(u;\theta_{0})$ are respectively defined by (\ref{definition_r0}) and
\begin{align*}
	&r_{T}^{(1)}(u;\theta_{0}):=\langle R_{T}(u;\theta_{0}),\zeta_{T}(\theta_{0})\rangle_{\mathbb{R}^2}
	=\langle \widetilde{R}_{T}(u;\theta_{0}),\Delta_{T}(\theta_{0})\rangle_{\mathbb{R}^2},\\ 
	&r_{T}^{(2)}(u;\theta_{0})
	:=-\langle R_{T}(u;\theta_{0}),\Gamma_{T}(\theta_{0})\varphi_{T}(\theta_{0})u\rangle_{\mathbb{R}^2}
	=-\langle\widetilde{R}_{T}(u;\theta_{0}),\widetilde{\Gamma}_{T}(\theta_{0})u\rangle_{\mathbb{R}^2},\\ 
	&r_{T}^{(3)}(u;\theta_{0})
	:=-\frac{1}{2}\langle R_{T}(u;\theta_{0}),\Gamma_{T}(\theta_{0})R_{T}(u;\theta_{0})\rangle_{\mathbb{R}^2}
	=-\frac{1}{2}\langle \widetilde{R}_{T}(u;\theta_{0}),\widetilde{\Gamma}_{T}(\theta_{0})\widetilde{R}_{T}(u;\theta_{0})\rangle_{\mathbb{R}^2}.
\end{align*}
Using (\ref{key_CLT}) and the Cauchy-Schwartz inequality, it suffices to prove that
\begin{equation*}
	\|\widetilde{R}_{T}(u;\theta_{0})\|_{\mathbb{R}^{2}}=o_{\mathbb{P}}(1)\ \ \mbox{as $T\to\infty$}
\end{equation*}
for any $u\in\mathbb{R}^{2}$. Using the properties of the Frobenius norm $\|\cdot\|_{F}$ and the operator norm $\|\cdot\|_{\mathrm{op}}$, we can show
\begin{align*}
	\|\widetilde{R}_{T}(u;\theta_{0})\|_{\mathbb{R}^{2}}
	&\leq\int_{0}^{1}\|\varphi_{T}(\theta_{0})^{-1}[J_{g}(\theta_{0}+\epsilon\Psi_{T}(\theta_{0})u)-J_{g}(\theta_{0})]
	\Psi_{T}(\theta_{0})u\|_{\mathbb{R}^{2}}\,\mathrm{d}\epsilon\\
	&\leq\|u\|_{\mathbb{R}^{2}}\int_{0}^{1}\|\varphi_{T}(\theta_{0})^{-1}[J_{g}(\theta_{0}+\epsilon\Psi_{T}(\theta_{0})u)-J_{g}(\theta_{0})]
	\Psi_{T}(\theta_{0})\|_{\mathrm{op}}\,\mathrm{d}\epsilon\\
	&\leq\|u\|_{\mathbb{R}^{2}}\|J_{g}(\theta_{0})^{-1}\|_{F}\int_{0}^{1}\|J_{g}(\theta_{0}+\epsilon\Psi_{T}(\theta_{0})u)-J_{g}(\theta_{0})\|_{F}\,\mathrm{d}\epsilon\\
	&\lesssim T^{H-1}+T^{-1/2}\lesssim T^{-1/2}.
\end{align*}
This completes the proof.

\subsection{Proof of Proposition~\ref{fOU_observed_Fisher}}\label{section_proof_preliminary_results}
Before stating the proof, we remark that a large part of the proof of Proposition \ref{fOU_observed_Fisher} is the same as the proof of Theorem~4.1 of \cite{Chiba-2020}. Notable exception is that the inequality (\ref{sufficient1}) holds also for $0<H\leq 1/4$. We note that this is due to {the} Gaussianity of the stationary solution $\bar{X}^{\theta}$. 

In the following, we can assume $\mathbb{E}[\bar{X}^{\theta_{0}}_t]=0$ without loss of generality. First we show the inequality (\ref{0928ineq1}). We define {the} analogue of $\langle M^{T}\rangle_{T}$ by
\[
\overline{\langle M^{T}\rangle}_{T} := \frac{1}{T}\int_{0}^{T}\beta_{t}[\mathrm{id}](\bar{\mathbf{X}}_{T}^{\theta_{0}})^{2}\,\mathrm{d}t,
\]
where $\bar{\mathbf{X}}_{T}^{\theta_{0}}:=(\bar{X}_{t}^{\theta_{0}})_{t\in[0,T]}$. Lemma 4.2 of \cite{Chiba-2020} gives {the} inequality
\begin{equation}\label{0406.ineq1}
	\mathbb{E}\left[ \left| \langle M^{T}\rangle_{T} - \overline{\langle M^{T}\rangle}_{T} \right|^{p} \right] \lesssim T^{-2pH}
\end{equation}
for any $p > 1$. We decompose $\beta_{t}[\mathrm{id}](\bar{\mathbf{X}}_{T}^{\theta_{0}})$ into {the following} two terms
\[
\beta_{t}[\mathrm{id}](\bar{\mathbf{X}}_{T}^{\theta_{0}}) = \sigma^{-1}\bar{d}_{H}^{-1}\int_{0}^{t} r^{-1/2-H}\bar{X}^{\theta_{0}}_{t-r}\,\mathrm{d}r + \sigma^{-1}\bar{d}_{H}^{-1}\int_{0}^{t} r^{-1/2-H}\left(\left(1-\frac{r}{t}\right)^{1/2-H}-1\right)\bar{X}^{\theta_{0}}_{t-r}\,\mathrm{d}r
\]
and denote the first and second term{s} by $\gamma^{1}_{t}[\mathrm{id}](\bar{\mathbf{X}}_{T}^{\theta_{0}})$ and $\gamma^{2}_{t}[\mathrm{id}](\bar{\mathbf{X}}_{T}^{\theta_{0}})$ {respectively}. Then $\overline{\langle M^{T}\rangle}_{T}$ can be {written} as
\[
\overline{\langle M^{T}\rangle}_{T} = \sum_{i,j=1,2}\frac{1}{T} \int_{0}^{T} \gamma^{i}_{t}[\mathrm{id}](\bar{\mathbf{X}}_{T}^{\theta_{0}})\gamma^{j}_{t}[\mathrm{id}](\bar{\mathbf{X}}_{T}^{\theta_{0}})\,\mathrm{d}t
\]
and {the Cauchy-Schwarz} 
inequality gives
\begin{align}
&\mathbb{E}\left[ \left|\overline{\langle M^{T}\rangle}_{T} -  \frac{1}{T} \int_{0}^{T} \gamma^{1}_{t}[\mathrm{id}](\bar{\mathbf{X}}_{T}^{\theta_{0}})^{{2}}\,\mathrm{d}t \right|^{{p}}\right] \nonumber\\
&{\lesssim} \sum_{(i,j)\neq(1,1)} \mathbb{E}\left[ \left|\frac{1}{T} \int_{0}^{T} \gamma^{i}_{t}[\mathrm{id}](\bar{\mathbf{X}}_{T}^{\theta_{0}})^{2}\,\mathrm{d}t \right|^{{p}}\right]^{1/2}
\mathbb{E}\left[ \left|\frac{1}{T} \int_{0}^{T} \gamma^{j}_{t}[\mathrm{id}](\bar{\mathbf{X}}_{T}^{\theta_{0}})^{2}\,\mathrm{d}t \right|^{{p}}\right]^{1/2} \label{0406.ineq2}
\end{align}
for any $p>1$. On the other hand, combining Proposition~3.3, Lemmas~4.5 and 4.6 {of} \cite{Chiba-2020}, we have the inequalities
\begin{equation}\label{0406.ineq3}
	\left|\mathbb{E}\left[\frac{1}{T} \int_{0}^{T} \gamma^{1}_{t}[\mathrm{id}](\bar{\mathbf{X}}_{T}^{\theta_{0}})^{2}\,\mathrm{d}t \right] - \frac{1}{2\alpha_{0}}\right| \lesssim T^{-2H}
\end{equation}
and
\begin{equation}\label{0406.ineq4}
	\left|\mathbb{E}\left[\frac{1}{T} \int_{0}^{T} \gamma^{2}_{t}[\mathrm{id}](\bar{\mathbf{X}}_{T}^{\theta_{0}})^{2}\,\mathrm{d}t \right] \right| \lesssim T^{-2H}.
\end{equation}
Therefore {
(\ref{0928ineq1}) follows from (\ref{0406.ineq1}), (\ref{0406.ineq2}), (\ref{0406.ineq3}) and (\ref{0406.ineq4}) once we have proved
}
that there exists some positive constant $\kappa(H)>0$ such that
        \begin{equation}\label{sufficient1}
		\mathbb{E}\left[\left|\frac{1}{T} \int_{0}^{T} \gamma^{i}_{t}[\mathrm{id}](\bar{\mathbf{X}}_{T}^{\theta_{0}})^{2}\,\mathrm{d}t - \mathbb{E}\left[\frac{1}{T} \int_{0}^{T} \gamma^{i}_{t}[\mathrm{id}](\bar{\mathbf{X}}_{T}^{\theta_{0}})^{2}\,\mathrm{d}t\right] \right|^{p}\right]\lesssim T^{-\kappa(H)p}
	\end{equation}
for $i=1,2$ and any $p>1$. {
In the following, we prove (\ref{sufficient1}) only in the case $i=1$ because (\ref{sufficient1}) in the case $i = 2$ can be proved in the similar way to the proof of (\ref{sufficient1}) in the case $i = 1$ using the inequality $\left|\left(1-r/t\right)^{1/2-H}-1\right|\leq1$ for $0\leq r \leq t$.
}

\underline{The case $p=2$.} We first prove (\ref{sufficient1}) when $p=2$. By a straight forward calculation, we can show
	\begin{align*}
		&\mathrm{Var}\left[\frac{{1}}{T}\int_{0}^{T} \gamma^{1}_{t}[\mathrm{id}](\bar{\mathbf{X}}_{T}^{\theta_{0}})^{2}\,\mathrm{d}t\right]\\
		&{\lesssim}\mathbb{E}\left[\left|\frac{1}{T}\int_{0}^T\,\mathrm{d}t\int_{0}^t\,\mathrm{d}s\int_{0}^t\,\mathrm{d}u\,(t-s)^{-H-\frac{1}{2}}(t-u)^{-H-\frac{1}{2}}\left(\bar{X}^{\theta_{0}}_{s}\bar{X}^{\theta_{0}}_{u}-\mathbb{E}[\bar{X}^{\theta_{0}}_{s}\bar{X}^{\theta_{0}}_{u}]\right)\right|^{2}\right]\\
		&=\frac{1}{T^{2}}\int_{0}^T\,\mathrm{d}t_{1}\int_{0}^T\,\mathrm{d}t_{2}
		\int_{0}^{t_{1}}\,\mathrm{d}s_{1}\int_{0}^{t_{1}}\,\mathrm{d}u_1\int_{0}^{t_{2}}\,\mathrm{d}s_{2}\int_{0}^{t_{2}}\,\mathrm{d}u_2\\
		&\hspace{2cm}(t_{1}-s_{1})^{-H-\frac{1}{2}}(t_{1}-u_1)^{-H-\frac{1}{2}}(t_{2}-s_{2})^{-H-\frac{1}{2}}(t_{2}-u_2)^{-H-\frac{1}{2}}
		\mathrm{Cov}[\bar{X}^{\theta_{0}}_{s_{1}}\bar{X}^{\theta_{0}}_{u_1},\bar{X}^{\theta_{0}}_{s_{2}}\bar{X}^{\theta_{0}}_{u_2}].
	\end{align*}
	Since $\bar{X}^{\theta_{0}}$ is a centered stationary Gaussian process, the Wick formula gives
	\begin{equation*}
		\mathrm{Cov}[\bar{X}^{\theta_{0}}_{s_{1}}\bar{X}^{\theta_{0}}_{u_1},\bar{X}^{\theta_{0}}_{s_{2}}\bar{X}^{\theta_{0}}_{u_2}]
		=c(s_{1}-s_{2})c(u_{1}-u_{2})+c(s_{1}-u_{2})c(u_{1}-s_{2}),
	\end{equation*}
	where $c(t)=\mathbb{E}[\bar{X}^{\theta_{0}}_t\bar{X}^{\theta_{0}}_0]$ for $t\in\mathbb{R}$. 
	Then we have
	\begin{align*}
		&\mathrm{Var}\left[\frac{1}{T}\int_{0}^T{\gamma_{t}^{1}}[\mathrm{id}](\bar{\mathbf{X}}_{T}^{\theta_{0}})^2\,\mathrm{d}t\right]\\
		&\lesssim\frac{1}{T^{2}}\int_{0}^T\,\mathrm{d}t_{1}\int_{0}^T\,\mathrm{d}t_{2}
		\left(\int_{0}^{t_{1}}\,\mathrm{d}s_{1}\int_{0}^{t_{2}}\,\mathrm{d}s_{2}\,(t_{1}-s_{1})^{-H-\frac{1}{2}}(t_{2}-s_{2})^{-H-\frac{1}{2}}c(s_{1}-s_{2})\right)^2\\
		&=\frac{2}{T^{2}}\int_{0}^T\,\mathrm{d}t_{1}\int_{0}^{t_{1}}\,\mathrm{d}t_{2}
		\left(\int_{0}^{t_{1}}\,\mathrm{d}s_{1}\int_{0}^{t_{2}}\,\mathrm{d}s_{2}\,(t_{1}-s_{1})^{-H-\frac{1}{2}}(t_{2}-s_{2})^{-H-\frac{1}{2}}c(s_{1}-s_{2})\right)^2\\
		&=\frac{2}{T^{2}}\left(\int_{0}^1\,\mathrm{d}t_{1}\int_{0}^{t_{1}}\,\mathrm{d}t_{2} +\int_{1}^T\,\mathrm{d}t_{1}\int_{0}^{t_{1}-1}\,\mathrm{d}t_{2} +\int_{1}^T\,\mathrm{d}t_{1}\int_{t_{1}-1}^{t_{1}}\,\mathrm{d}t_{2}\right)\\
		&\hspace{2cm}\left(\int_{0}^{t_{1}}\,\mathrm{d}s_{1}\int_{0}^{t_{2}}\,\mathrm{d}s_{2}\,(t_{1}-s_{1})^{-H-\frac{1}{2}}(t_{2}-s_{2})^{-H-\frac{1}{2}}c(s_{1}-s_{2})\right)^2\\
		&=:I_1(T)+I_2(T)+I_3(T).
	\end{align*}
Here $I_{1}(T)$ is obviously $O(T^{-2})$ as $T\to\infty$. First we bound $I_{3}(T)$.
Since $c(t)=O(|t|^{2H-2})$ as $|t|\to\infty$,
\begin{equation}\label{covariance_integrability}
	\int_{-\infty}^{\infty}|c(t)|^{q}\,\mathrm{d}t<\infty
\end{equation}
for any $q\geq 1$. Let $r\in(1,(1/2+H)^{-1})$ and $q:=r/(1-r)$. Using the H\"older inequality,
\begin{align}
	\nonumber
	&\int_{0}^{t_{1}}\,\mathrm{d}s_{1}\int_{0}^{t_{2}}\,\mathrm{d}s_{2}\,(t_{1}-s_{1})^{-H-\frac{1}{2}}(t_{2}-s_{2})^{-H-\frac{1}{2}}|c(s_{1}-s_{2})|\\
	\nonumber
	&\lesssim\int_{0}^{t_{2}}\,\mathrm{d}s_{2}\,(t_{2}-s_{2})^{-H-\frac{1}{2}}
	\left(\int_{0}^{t_{1}}(t_{1}-s_{1})^{-r(H+\frac{1}{2})}\,\mathrm{d}s_{1}\right)^{\frac{1}{r}}
	\left(\int_{0}^{t_{1}}|c(s_{1}-s_{2})|^{q}\,\mathrm{d}s_{1}\right)^{\frac{1}{q}}\\
	\label{inequality_Fisher1}
	&\lesssim\int_{0}^{t_{2}}\,\mathrm{d}s_{2}\,(t_{2}-s_{2})^{-H-\frac{1}{2}}
	t_{1}^{\frac{1}{r}-(H+\frac{1}{2})}\left(\int_{-\infty}^{\infty}|c(t)|^{q}\,\mathrm{d}t\right)^{\frac{1}{q}}
	\lesssim t_{2}^{\frac{1}{2}-H}t_{1}^{\frac{1}{r}-(H+\frac{1}{2})}
\end{align}
so that we obtain
\begin{align*}
	I_{3}(T)\lesssim
	\frac{1}{T^{2}}\int_{1}^T\,\mathrm{d}t_{1}\int_{t_{1}-1}^{t_{1}}\,\mathrm{d}t_{2}\,\left(t_{2}^{\frac{1}{2}-H}t_{1}^{\frac{1}{r}-(H+\frac{1}{2})}\right)^{2}
	\lesssim\frac{1}{T^{2}}\int_{1}^T\,\mathrm{d}t_{1}\,t_{1}^{\frac{2}{r}-(2H+1)+1-2H}
	\lesssim T^{\frac{2}{r}-(4H+1)}.
\end{align*}
Next we bound $I_{2}(T)$ as follows: 
\begin{align*}
	I_{2}(T)
	&\lesssim\frac{1}{T^{2}}\int_{1}^T\,\mathrm{d}t_{1}\int_{0}^{t_{1}-1}\,\mathrm{d}t_{2}
	\left\{\left(\int_{0}^{t_{2}}\,\mathrm{d}s_{1}\int_{0}^{t_{2}}\,\mathrm{d}s_{2}\,(t_{1}-s_{1})^{-H-\frac{1}{2}}(t_{2}-s_{2})^{-H-\frac{1}{2}}|c(s_{1}-s_{2})|\right)^2\right.\\
	&\hspace{2cm}\left.+\left(\int_{t_{2}}^{t_{1}}\,\mathrm{d}s_{1}\int_{0}^{t_{2}}\,\mathrm{d}s_{2}\,(t_{1}-s_{1})^{-H-\frac{1}{2}}(t_{2}-s_{2})^{-H-\frac{1}{2}}c(|s_{1}-s_{2}|)\right)^2\right\}\\
	&=:I_{2,1}(T)+I_{2,2}(T).
\end{align*}
First we bound $I_{2,1}(T)$. Using (\ref{covariance_integrability}), we can show
\begin{align*}
	I_{2,1}(T)&\leq
	\frac{1}{T^{2}}\int_{1}^T\,\mathrm{d}t_{1}\int_{0}^{t_{1}-1}\,\mathrm{d}t_{2}
	\left((t_{1}-t_{2})^{-H-\frac{1}{2}}\int_{0}^{t_{2}}\,\mathrm{d}s_{2}\,(t_{2}-s_{2})^{-H-\frac{1}{2}}\int_{0}^{t_{2}}\,\mathrm{d}s_{1}|c(s_{1}-s_{2})|\right)^2\\
	&\lesssim
	\frac{1}{T^{2}}\int_{1}^T\,\mathrm{d}t_{1}\int_{0}^{t_{1}-1}\,\mathrm{d}t_{2}
	\left((t_{1}-t_{2})^{-H-\frac{1}{2}}t_{2}^{\frac{1}{2}-H}\int_{-\infty}^{\infty}|c(v_{1})|\,\mathrm{d}v_{1}\right)^2\\
	&\lesssim
	\frac{1}{T^{2}}\int_{1}^T\,\mathrm{d}t_{1}t_{1}^{1-2H}\int_{0}^{t_{1}-1}\,\mathrm{d}t_{2}~(t_{1}-t_{2})^{-2H-1}\\
	&\lesssim
	\frac{1}{T^{2}}\int_{1}^T\,\mathrm{d}t_{1}t_{1}^{1-2H}\left(1+t_{1}^{-2H}\right)
	\lesssim T^{-2H}+T^{-4H}\lesssim T^{-2H}.
\end{align*}
Finally we bound $I_{2,2}(T)$. Using (\ref{inequality_Fisher1}), we can show
\begin{align*}
	I_{2,2}(T)
	&\lesssim\frac{1}{T^{2}}\int_{1}^T\,\mathrm{d}t_{1}\int_{0}^{t_{1}}\,\mathrm{d}t_{2}\,t_{2}^{\frac{1}{2}-H}t_{1}^{\frac{1}{r}-(H+\frac{1}{2})}\\
	&\hspace{2cm}\times\int_{t_{2}}^{t_{1}}\,\mathrm{d}s_{1}\int_{0}^{t_{2}}\,\mathrm{d}s_{2}\,(t_{1}-s_{1})^{-H-\frac{1}{2}}(t_{2}-s_{2})^{-H-\frac{1}{2}}|c(s_{1}-s_{2})|\\
	&\leq\frac{1}{T^{2}}\int_{1}^T\,\mathrm{d}t_{1}\,t_{1}^{\frac{1}{r}-2H}\int_{0}^{t_{1}}\,\mathrm{d}t_{2}\int_{t_{2}}^{t_{1}}\,\mathrm{d}s_{1}\int_{0}^{t_{2}}\,\mathrm{d}s_{2}\,(t_{1}-s_{1})^{-H-\frac{1}{2}}(t_{2}-s_{2})^{-H-\frac{1}{2}}|c(s_{1}-s_{2})|\\
	&=\frac{1}{T^{2}}\int_{1}^T\,\mathrm{d}t_{1}\,t_{1}^{\frac{1}{r}-2H}
	\int_{0}^{t_{1}}\,\mathrm{d}s_{1}\int_{0}^{s_{1}}\,\mathrm{d}s_{2}\int_{s_{2}}^{s_{1}}\,\mathrm{d}t_{2}\,(t_{1}-s_{1})^{-H-\frac{1}{2}}(t_{2}-s_{2})^{-H-\frac{1}{2}}|c(s_{1}-s_{2})|\\
	&=\frac{1}{T^{2}}\int_{1}^T\,\mathrm{d}t_{1}\,t_{1}^{\frac{1}{r}-2H}
	\int_{0}^{t_{1}}\,\mathrm{d}s_{1}\,(t_{1}-s_{1})^{-H-\frac{1}{2}}
	\int_{0}^{s_{1}}\,\mathrm{d}s_{2}\,(s_{1}-s_{2})^{\frac{1}{2}-H}|c(s_{1}-s_{2})|\\
	&\lesssim\frac{1}{T^{2}}\int_{1}^T\,\mathrm{d}t_{1}\,t_{1}^{\frac{1}{r}-2H}
	\int_{0}^{t_{1}}\,\mathrm{d}s_{1}\,(t_{1}-s_{1})^{-H-\frac{1}{2}},
\end{align*}
where we used the fact that
\begin{equation*}
	\int_{0}^{s_{1}}\,\mathrm{d}s_{2}\,(s_{1}-s_{2})^{\frac{1}{2}-H}|c(s_{1}-s_{2})|
	\leq\int_{0}^{\infty}\,\mathrm{d}v_{2}\,v_{2}^{\frac{1}{2}-H}|c(v_{2})|<\infty
\end{equation*}
thanks to
$c(t)=O(|t|^{2H-2})$ as $|t|\to\infty$. Therefore we obtain
\begin{equation*}
	I_{2,2}(T)
	\lesssim\frac{1}{T^{2}}\int_{1}^T\,\mathrm{d}t_{1}\,t_{1}^{\frac{1}{r}+\frac{1}{2}-3H}
	\lesssim T^{\frac{1}{r}-\frac{1}{2}-3H}.
\end{equation*}
To sum up, we obtain the following upper bound:
\begin{align} \label{0927ineq1}
	\mathrm{Var}\left[\frac{1}{T}\int_{0}^{T} \gamma^{1}_{t}[\mathrm{id}](\bar{\mathbf{X}}_{T}^{\theta_{0}})^{2}\,\mathrm{d}t \right]
	\lesssim T^{-2}+T^{\frac{2}{r}-(4H+1)}+T^{-2H}+T^{\frac{1}{r}-\frac{1}{2}-3H}.
\end{align}
Therefore (\ref{sufficient1}) follows if we take $r$ satisfying
\begin{equation*}
	1\vee\frac{2}{4H+1}<r<(H+1/2)^{-1}.
\end{equation*}
This completes the proof. 

\underline{{The general case $p>1$.}} Since the {random} variable
\[
\frac{1}{T}\int_{0}^{T} \gamma^{1}_{t}[\mathrm{id}](\bar{\mathbf{X}}_{T}^{\theta_{0}})^{2}\,\mathrm{d}t - \mathbb{E}\left[\frac{1}{T}\int_{0}^{T} \gamma^{1}_{t}[\mathrm{id}](\bar{\mathbf{X}}_{T}^{\theta_{0}})^{2}\,\mathrm{d}t\right]
\]
is {an element of} the second-order Wiener chaos, the hypercontractivity of the Wiener integral gives the inequality
\begin{align*}
&\mathbb{E}\left[\left|\frac{1}{T}\int_{0}^{T} \gamma^{1}_{t}[\mathrm{id}](\bar{\mathbf{X}}_{T}^{\theta_{0}})^{2}\,\mathrm{d}t - \mathbb{E}\left[\frac{1}{T}\int_{0}^{T} \gamma^{1}_{t}[\mathrm{id}](\bar{\mathbf{X}}_{T}^{\theta_{0}})^{2}\,\mathrm{d}t\right] \right|^{p}\right] \leq C(p) \mathrm{Var}\left[\frac{1}{T}\int_{0}^{T} \gamma^{1}_{t}[\mathrm{id}](\bar{\mathbf{X}}_{T}^{\theta_{0}})^{2}\,\mathrm{d}t\right]
^{\frac{p}{2}}
\end{align*}
for some positive constants $C(p) >0$, see Section 2.7 of \cite{Nourdin-Peccati-2012-book}. 
{Since the right hand side of the above inequality is further dominated using the inequality (\ref{0927ineq1}), we finish the proof of (\ref{sufficient1}).
}

Finally we prove the inequality (\ref{0928ineq2}).
The inequality (\ref{0928ineq2}) follows directly from (4.2) {in} Theorem~4.1 {of} \cite{Chiba-2020}. 
Indeed, (4.2) {in} Theorem~4.1 {of} \cite{Chiba-2020} gives 
\[
\mathbb{E}\left[\left| \int_{0}^{T}t^{1/2-H}\beta_{t}[\mathrm{id}](\mathbf{X}^{\theta_{0}}_{T})\,\mathrm{d}t\right|^{p}\right] \lesssim T^{-p(2H-3/2)}
\]
and hence 
\[
\mathbb{E}\left[\left|\langle M^{T},N^{T}\rangle_{T}\right|^{p}\right] \lesssim T^{p(H-3/2)} T^{-p(2H-3/2)} = T^{-pH}
\]
for any $p > 1$. This completes the proof.

\bibliographystyle{acmtrans-ims}
\bibliography{myref}

\end{document}